\documentclass[11pt]{article}

\usepackage[T1]{fontenc}
\usepackage{lmodern}
\usepackage{microtype}
\usepackage[a4paper,margin=25mm,headheight=14pt]{geometry}
\usepackage{amsmath,amssymb,amsthm,mathtools}
\usepackage{enumitem}
\usepackage{booktabs}
\usepackage{xcolor}
\usepackage[hidelinks]{hyperref}
\usepackage{fancyhdr}

\definecolor{linkblue}{RGB}{28,73,122}
\hypersetup{
  colorlinks=true,
  linkcolor=linkblue,
  citecolor=linkblue,
  urlcolor=linkblue,
  pdftitle={A thickness boundary and modular obstructions for two-set radial projections},
  pdfauthor={Miwa Yuuki}
}

\setlist{nosep,leftmargin=1.6em}
\numberwithin{equation}{section}

\newtheorem{theorem}{Theorem}[section]
\newtheorem{proposition}[theorem]{Proposition}
\newtheorem{corollary}[theorem]{Corollary}
\newtheorem{lemma}[theorem]{Lemma}
\theoremstyle{remark}
\newtheorem{remark}[theorem]{Remark}

\newcommand{\R}{\mathbb R}
\newcommand{\T}{\mathbb T}
\newcommand{\Z}{\mathbb Z}
\newcommand{\Hd}{\dim_{\mathrm H}}
\newcommand{\Ldim}{\dim_{l^1}}
\newcommand{\conv}{\operatorname{conv}}

\begin{document}

\begin{center}
  {\LARGE\bfseries
  A thickness boundary and modular obstructions\\[0.12em]
  for two-set radial projections\par}
  \vspace{0.3em}
  {\large Miwa Yuuki\qquad\small 3 September 2026\par}
\end{center}
\vspace{0.2em}

\begin{abstract}
Let \(K_{a,m}\) and \(K_{b,n}\) be missing-digit Cantor sets with
initial consecutive digit sets, and write
\(c_a=m/(a-1)\) and \(c_b=n/(b-1)\).
We prove that if \(c_a+c_b\geq1\), then the radial projection of
\(K_{a,m}\times K_{b,n}\) from every observer has nonempty interior;
no multiplicative-independence assumption is needed for this implication.
Conversely, when the bases are multiplicatively independent and
\(c_a+c_b<1\), we exhibit explicit unbounded open sets of observers for
which the radial image is compact and nowhere dense.  Rational observers
with the same property are dense in each exterior corner region and
occur arbitrarily close to the four corners of the product.  Thus
\(c_a+c_b=1\) is the exact threshold for the all-observers interior
property within the multiplicatively independent initial-block family.
In particular, this supplies an explicit two-set counterexample to the
nonempty-interior conclusions of two conjectures of Yu.
We also establish a sufficient modular phase obstruction for affine
translates of such sets.  Combining it with a fixed-pin positive-measure
theorem of Banaji and Yu yields consecutive-block division sets that are
compact, perfect, of positive Lebesgue measure, and nowhere dense, with
both factor dimensions tending to one.  For the same family, Fourier
\(l^1\)-dimension estimates and the incidence argument used in Yu's
product theorem imply that the two self-products contain intervals for
all sufficiently large \(r\) and \(s\), while the cross-division set
remains nowhere dense.
\end{abstract}

\section{Introduction and main results}

For integers \(a\geq3\) and \(1\leq m\leq a-2\), put
\begin{equation}\label{eq:initial-set}
 K_{a,m}
 =
 \left\{\sum_{j=1}^{\infty}u_ja^{-j}:
 u_j\in\{0,\ldots,m\}\right\},
 \qquad
 c_a=\frac{m}{a-1}.
\end{equation}
Thus \(K_{a,m}\subset[0,c_a]\).  Define \(K_{b,n}\) and
\(c_b=n/(b-1)\) analogously.  For an observer \(O\in\R^2\), write
\[
 \Pi_O(z)=\frac{z-O}{|z-O|}\in S^1
\]
for radial projection, with \(O\) omitted from the domain if it belongs
to the set being projected.

Our first result identifies the all-observers interior threshold in this
family.

\begin{theorem}[Thickness boundary]\label{thm:boundary}
Let \(a,b\geq3\), \(1\leq m\leq a-2\), and
\(1\leq n\leq b-2\).
\begin{enumerate}[label=(\roman*)]
\item If \(c_a+c_b\geq1\), then
\[
 \Pi_O(K_{a,m}\times K_{b,n})
\]
has nonempty interior for every \(O\in\R^2\).  This implication does not
require multiplicative independence.
\item If \(\log a/\log b\notin\mathbb Q\) and \(c_a+c_b<1\), then the
all-observers conclusion fails.  In fact, an explicit unbounded open
set of observers sees compact nowhere-dense radial images.  Moreover,
rational observers with this property are dense in each of the four
components of
\[
 \bigl(\R\setminus[0,c_a]\bigr)
 \times
 \bigl(\R\setminus[0,c_b]\bigr).
\]
\end{enumerate}
Consequently, within the multiplicatively independent initial-block
family,
\[
 c_a+c_b\geq1
 \quad\Longleftrightarrow\quad
 \text{every observer sees radial interior}.
\]
\end{theorem}

The positive direction is a local application of Astels' interval
criterion on appropriately chosen corner cylinders.  The negative
direction comes from a modular digit-tail obstruction.  A first form of
that obstruction applies to affine translates
\[
 E=\alpha+K_{a,m},
 \qquad
 F=\beta+K_{b,n},
\]
provided the two anchor residues stabilize under multiplication by
powers of their bases.  It produces explicit division sets \(F/E\)
with empty interior.

The simplest initial-block consequence is an explicit two-set
counterexample to the nonempty-interior conclusions proposed in
\cite[Conjectures~2.1 and 4.2]{YuRadial}.
In both conjectures this conclusion is quantified over every observer.

\begin{corollary}\label{cor:small-counterexample}
The radial image
\[
 \Pi_{(-1,-1)}(K_{3,1}\times K_{5,1})
\]
is compact and nowhere dense, although
\[
 \Hd K_{3,1}+\Hd K_{5,1}
 =
 \frac{\log2}{\log3}+\frac{\log2}{\log5}>1
\]
and the two contraction ratios are multiplicatively independent.
\end{corollary}

Indeed, both defining iterated function systems satisfy the strong
separation condition, hence the open set condition; their uniform
contraction ratios \(1/3\) and \(1/5\) are multiplicatively independent;
and the displayed Hausdorff dimension is greater than one.  Thus
Corollary~\ref{cor:small-counterexample} disproves
\cite[Conjecture~2.1]{YuRadial} and the nonempty-interior assertion of
\cite[Conjecture~4.2]{YuRadial}, as stated.  No conclusion about the
positive-measure assertion of Conjecture~4.2 is intended.

The obstruction is not confined to a single arithmetic observer.  If
\(\gamma=1-c_a-c_b>0\), define
\begin{equation}\label{eq:arm-thresholds-intro}
 A_*=\frac{c_a(c_a+c_b)}{2\gamma},
 \qquad
 B_*=\frac{c_b(c_a+c_b)}{2\gamma}.
\end{equation}
When the bases are multiplicatively independent, every observer in
\[
 \{(x,y):\operatorname{dist}(x,[0,c_a])>A_*,\ y\notin[0,c_b]\}
 \ \cup\
 \{(x,y):x\notin[0,c_a],\ \operatorname{dist}(y,[0,c_b])>B_*\}
\]
sees a compact nowhere-dense image.
In fact, rational bad observers are dense in all four exterior corner
regions.  Unlike the open arms, this arithmetic family approaches the
product arbitrarily closely at each of its four corners.  For
\(K_{3,1}\times K_{5,1}\), the open Euclidean ball of radius \(5/8\)
centred at \((-1,-1)\) consists entirely of bad observers.

Affine anchors also give a family in which nowhere denseness coexists
with positive measure.  For \(r,s\geq1\), set
\[
 A_r=25^r,\qquad B_s=49^s,\qquad
 q_r=\frac{A_r-1}{24},\qquad t_s=\frac{B_s-1}{24},
\]
and
\begin{equation}\label{eq:contrast-family-intro}
 E_r=K_{A_r,\{13q_r,\ldots,24q_r\}},
 \qquad
 F_s=K_{B_s,\{t_s,\ldots,12t_s\}}.
\end{equation}

\begin{theorem}[Product--division contrast]\label{thm:contrast-intro}
For every \(r,s\geq1\), the division set \(F_s/E_r\) is compact,
perfect, nowhere dense, and of positive Lebesgue measure.  Moreover,
\[
 \Hd E_r\longrightarrow1,\qquad \Hd F_s\longrightarrow1.
\]
For all sufficiently large \(r,s\), each self-product
\[
 E_rE_r\qquad\text{and}\qquad F_sF_s
\]
contains a nondegenerate interval.  The explicit estimate below already
shows that \(E_6E_6\) and \(F_5F_5\) contain nondegenerate intervals.
No assertion is made about the cross-product \(E_rF_s\).
\end{theorem}

These statements complement the explicit arithmetic of the classical
Cantor quotient in \cite{ART}, the thickness criteria of
\cite{Astels}, and the positive-measure results of \cite{BanajiYu}.
The point here is not that every choice of anchors for a thin pair has a
nowhere-dense quotient: the modular criteria below are sufficient
conditions and do not characterize all anchors.

\section{Preliminaries}\label{sec:prelim}

\subsection{Digit blocks, tails, and thickness}

For a digit set \(D\subset\{0,\ldots,a-1\}\), write
\[
 K_{a,D}
 =
 \left\{\sum_{j=1}^{\infty}d_ja^{-j}:d_j\in D\right\}.
\]
If \(D=\{d,\ldots,d+m\}\) is consecutive, then
\begin{equation}\label{eq:affine-block}
 K_{a,D}
 =
 \frac{d}{a-1}+K_{a,m}.
\end{equation}
For the initial block in \eqref{eq:initial-set},
\begin{equation}\label{eq:geometry}
 \conv(K_{a,m})=[0,c_a],
 \qquad
 K_{a,m}=c_a-K_{a,m},
 \qquad
 g_a=\frac{1-c_a}{a},
\end{equation}
where \(g_a\) is the largest bounded gap.  Its normalized thickness is
\begin{equation}\label{eq:normalized-thickness}
 S(K_{a,m})=c_a;
\end{equation}
see \cite[Lemma~3.6]{YuRadial}.  Astels' theorem includes the equality
case \(S(K_1)+S(K_2)=1\), subject to the strict gap-size conditions used
below \cite[Theorems~2.2 and 2.4]{Astels}.

Multiplication by a base power shifts the digit expansion:
\begin{lemma}[Digit tails]\label{lem:tails}
If \(u\in K_{a,m}\) and \(p\geq1\), then
\[
 \{a^pu\}\in K_{a,m},
\]
where \(\{\cdot\}\) denotes fractional part.
\end{lemma}
\begin{proof}
The integer part is determined by the first \(p\) digits and the
fractional part is the remaining tail.
\end{proof}

The anchor in \eqref{eq:affine-block} has a fixed residue:
\begin{equation}\label{eq:anchor-residue}
 a^p\frac{d}{a-1}-\frac{d}{a-1}
 =
 d(1+a+\cdots+a^{p-1})\in\Z.
\end{equation}

\subsection{Phase density and slope coordinates}

\begin{lemma}[Vanishing perturbations of an irrational rotation]
\label{lem:phase-density}
If \(\xi\notin\mathbb Q\) and \(\varepsilon_p\to0\), then every tail of
\[
 p\xi+\varepsilon_p\pmod1
\]
is dense in \(\T=\R/\Z\).
\end{lemma}
\begin{proof}
Given an open arc \(I\subset\T\), choose an open arc \(J\) with
\(\overline J\subset I\).  The irrational rotation visits \(J\) at
arbitrarily large indices, and every sufficiently late such visit
remains in \(I\) after adding \(\varepsilon_p\).
\end{proof}

\begin{lemma}[Scale density]\label{lem:scale-density}
Let \(a,b>1\) and suppose that \(\log a/\log b\notin\mathbb Q\).  Then,
for every \(N\in\mathbb N\),
\[
 \left\{\frac{b^q}{a^p}:
 p,q\in\mathbb Z,\ p,q\geq N\right\}
\]
is dense in \(\mathbb R_{>0}\).
\end{lemma}

\begin{proof}
Put \(\theta=\log_b a\), which is irrational.  Fix \(t\in\mathbb R\)
and \(\varepsilon>0\).  Every tail of the irrational rotation
\(p\theta+t\pmod1\) is dense, so there are arbitrarily large \(p\) such
that
\[
 \operatorname{dist}(p\theta+t,\mathbb Z)<\varepsilon.
\]
Choose \(q\in\mathbb Z\) with
\(|q-(p\theta+t)|<\varepsilon\).  Taking \(p\) sufficiently large
ensures \(p,q\geq N\), and then
\[
 \left|\log_b\!\left(\frac{b^q}{a^p}\right)-t\right|
 =|q-p\theta-t|<\varepsilon.
\]
Thus the logarithms of the displayed scales are dense in \(\mathbb R\).
Exponentiation by base \(b\) proves the claim.
\end{proof}

If an observer is southwest of a compact subset of the positive
quadrant, all relevant directions lie in a compact positive arc of
\(S^1\).  On this arc the angle, slope \(y/x\), and logarithmic slope
\(\log_b(y/x)\) are smooth coordinates.  They are local
homeomorphisms for topological statements and bi-Lipschitz on compact
subarcs for Lebesgue-measure statements.  For an arbitrary observer we
will use the ordinary slope only on a small corner rectangle whose first
coordinate stays away from the observer.  This distinction is important:
if an observer lies in the product, removing it can destroy compactness,
so no compactness assertion is made in the thick-side part of
Theorem~\ref{thm:boundary}.

\section{A modular obstruction for affine translates}
\label{sec:affine}

Let
\[
 E=\alpha+K_{a,m},
 \qquad
 F=\beta+K_{b,n},
\]
where \(\alpha,\beta>0\).  Put \(c_b=n/(b-1)\), and assume
\(\log a/\log b\notin\mathbb Q\).  Suppose that for all sufficiently
large \(p,q\),
\begin{equation}\label{eq:fixed-residues}
 \{a^p\alpha\}=\alpha_0,
 \qquad
 \{b^q\beta\}=\beta_0.
\end{equation}
Choose a real lift
\(\delta=\alpha_0-\beta_0\).

\begin{theorem}[Affine-anchor obstruction]\label{thm:affine}
Suppose that there are \(e\in\R\) and \(k\in\Z\) such that
\begin{equation}\label{eq:affine-window}
 k+c_b<L_e\leq U_e<k+1,
\end{equation}
where
\begin{equation}\label{eq:affine-hull}
 L_e=\delta+e\alpha+\min(0,ec_a),
 \qquad
 U_e=\delta+e\alpha+c_a+\max(0,ec_a).
\end{equation}
Then the division set
\[
 F/E=\left\{\frac{\beta+v}{\alpha+u}:
 u\in K_{a,m},\ v\in K_{b,n}\right\}
\]
is compact and nowhere dense.
\end{theorem}

\begin{proof}
Consider
\[
 \eta_p=\log_b(a^p+e)\pmod1.
\]
For all sufficiently large \(p\), the argument of the logarithm is
positive, and
\[
 \eta_p
 =
 p\frac{\log a}{\log b}
 +\log_b(1+ea^{-p})\pmod1.
\]
Lemma~\ref{lem:phase-density} shows that every tail is dense.

Suppose that a sufficiently late \(\eta_p\) belongs to
\(\log_b(F/E)\bmod1\).  Then for some integer \(q\) and some
\(u\in K_{a,m}\), \(v\in K_{b,n}\),
\begin{equation}\label{eq:affine-equality}
 (a^p+e)(\alpha+u)=b^q(\beta+v).
\end{equation}
Because \(E\) and \(F\) are compact and bounded away from zero,
\(q=p\log_ba+O(1)\), so \(q\to\infty\).
By Lemma~\ref{lem:tails},
\[
 \theta=\{a^pu\}\in K_{a,m},
 \qquad
 \zeta=\{b^qv\}\in K_{b,n}.
\]
Taking fractional parts in \eqref{eq:affine-equality} gives
\begin{equation}\label{eq:affine-fractional}
 \zeta
 \equiv
 \delta+\theta+e(\alpha+u)
 \pmod1.
\end{equation}
The real expression on the right lies in \([L_e,U_e]\), which
\eqref{eq:affine-window} places strictly inside one lift of the principal
gap \((c_b,1)\).  This contradicts
\(\zeta\in K_{b,n}\subset[0,c_b]\).

Thus a dense tail is absent from the compact phase image
\(\log_b(F/E)\bmod1\).  The phase image has empty interior.  Since the
quotient map \(\R\to\T\) is open and logarithm is a homeomorphism on
\((0,\infty)\), the compact set \(F/E\) has empty interior and is
nowhere dense.
\end{proof}

\begin{remark}\label{rem:affine-envelope}
A condition slightly finer than the hull test
\eqref{eq:affine-window}--\eqref{eq:affine-hull}, but still only
sufficient, is
\[
 \bigl(\delta+K_{a,m}+e(\alpha+K_{a,m})\bigr)\bmod1
 \ \cap\ K_{b,n}=\varnothing.
\]
The two displayed copies of \(K_{a,m}\) are independent: their
Minkowski sum is an upper envelope for the correlated pair
\((\{a^pu\},u)\) in \eqref{eq:affine-fractional}.
\end{remark}

For consecutive blocks the residue hypothesis follows from
\eqref{eq:anchor-residue}.  We will also need a common-anchor
consequence.

\begin{corollary}[Common anchors]\label{cor:common-anchor}
Assume \(\log a/\log b\notin\mathbb Q\), and let \(t>0\) satisfy
\[
 t(a-1),\ t(b-1)\in\Z.
\]
If
\begin{equation}\label{eq:common-anchor-condition}
 t(1-c_a-c_b)>\min(c_a,c_b)^2,
\end{equation}
then
\[
 \frac{t+K_{b,n}}{t+K_{a,m}}
\]
is compact and nowhere dense.
\end{corollary}

\begin{proof}
Here the two anchor residues coincide, so \(\delta=0\).  With \(e<0\),
the interval in
\eqref{eq:affine-hull} is
\[
 [\,e(t+c_a),\,et+c_a\,].
\]
It can be placed inside the lift \((-1+c_b,0)\) whenever
\[
 -\frac{1-c_b}{t+c_a}<e<-\frac{c_a}{t}.
\]
This interval is nonempty precisely when
\[
 t(1-c_a-c_b)>c_a^2.
\]
Interchanging the two coordinates and taking reciprocal quotients gives
the analogous condition with \(c_b^2\).  Taking the cheaper alternative
proves \eqref{eq:common-anchor-condition}.
\end{proof}

Corollary~\ref{cor:small-counterexample} follows at once:
\(c_a=1/2\), \(c_b=1/4\), \(t=1\), and
\[
 1\left(1-\frac12-\frac14\right)
 >
 \min\left(\frac12,\frac14\right)^2.
\]
Here the strict alternative is realized after interchanging the two
coordinates.  For the reciprocal quotient
\[
 \frac{1+K_{3,1}}{1+K_{5,1}},
\]
the direct window is
\[
 -\frac25<e<-\frac14.
\]
It is therefore compact and nowhere dense; inversion gives the actual
slope quotient \((1+K_{5,1})/(1+K_{3,1})\) seen from \((-1,-1)\).

\section{The all-observers thickness boundary}
\label{sec:thickness}

We now prove the positive half of Theorem~\ref{thm:boundary}.  Write
\[
 g_a=\frac{1-c_a}{a},
 \qquad
 g_b=\frac{1-c_b}{b}.
\]
Suppose \(c_a+c_b\geq1\).  Astels' interval theorem, applied to
\(K_{b,n}\) and \(-\rho K_{a,m}\), gives
\begin{equation}\label{eq:astels-difference}
 K_{b,n}-\rho K_{a,m}=[-\rho c_a,c_b]
\end{equation}
whenever
\begin{equation}\label{eq:astels-window}
 \rho\in W:=
 \left(\frac{g_b}{c_a},\frac{c_b}{g_a}\right).
\end{equation}
Indeed, \eqref{eq:astels-window} and the automatic inequalities
\(c_a>g_a\), \(c_b>g_b\) are exactly
\[
 \min(c_b,\rho c_a)>\max(g_b,\rho g_a),
\]
while \eqref{eq:normalized-thickness} supplies
\(S(K_{a,m})+S(K_{b,n})\geq1\).  This is where the equality case in
Astels' theorem is used.

\begin{proof}[Proof of Theorem~\ref{thm:boundary}, positive direction]
Fix \(O=(X,Y)\in\R^2\).  On the first axis choose
\(x_*\in\{0,c_a\}\) and its inward sign
\(\sigma_x\in\{+1,-1\}\), with \(+1\) at the left endpoint and
\(-1\) at the right endpoint, so that
\begin{equation}\label{eq:corner-choice}
 \sigma_x(x_*-X)<0,
 \qquad
 x_*-X\neq0.
\end{equation}
Use the left endpoint when \(X>0\) and the right endpoint when
\(X\leq0\).  Choose \(y_*,\sigma_y\) analogously.  Put
\[
 A=x_*-X,\qquad B=y_*-Y,\qquad r_0=\frac BA,
\]
so
\[
 \kappa:=\sigma_x\sigma_y r_0>0.
\]

The selected corner cylinders have the form
\[
 x=x_*+\sigma_xa^{-p}u,
 \qquad
 y=y_*+\sigma_yb^{-q}v,
\]
where \(u\in K_{a,m}\), \(v\in K_{b,n}\).  The equation that their
slope from \(O\) is \(r\) becomes
\begin{equation}\label{eq:corner-difference}
 v-\rho(r)u=t(r),
\end{equation}
with
\begin{equation}\label{eq:corner-parameters}
 \rho(r)=\sigma_x\sigma_y r\frac{b^q}{a^p},
 \qquad
 t(r)=\sigma_yA b^q(r-r_0).
\end{equation}

First suppose \(\log a/\log b\notin\mathbb Q\).  By
Lemma~\ref{lem:scale-density}, the scales \(b^q/a^p\), with \(p,q\)
arbitrarily large, are dense in \(\R_{>0}\).  Choose deep cylinders for
which
\[
 \rho(r_0)=\kappa\frac{b^q}{a^p}\in W.
\]
At \(r=r_0\), one has \(t(r_0)=0\), strictly inside
\([-\rho(r_0)c_a,c_b]\).  By continuity, for every \(r\) in a small
open interval about \(r_0\), both \(\rho(r)\in W\) and
\[
 t(r)\in[-\rho(r)c_a,c_b].
\]
Equations \eqref{eq:astels-difference} and
\eqref{eq:corner-difference} realize every such slope.  Taking the
cylinders deep keeps their first coordinates away from \(X\), so this
is an open arc in the radial image.  The construction remains valid when
\(O\in K_{a,m}\times K_{b,n}\), because the selected deep corner
rectangle avoids \(O\).

If the bases are multiplicatively dependent, write
\[
 a=A_0^r,\qquad b=A_0^s,\qquad \gcd(r,s)=1.
\]
The available relative scales are \(A_0^k\).  Bézout gives an integer
solution of \(sq-rp=k\), and adding
\((p,q)\mapsto(p+st,q+rt)\) makes both exponents arbitrarily large.
Moreover,
\[
 \frac{\sup W}{\inf W}
 =
 ab\,\frac{c_a}{1-c_a}\frac{c_b}{1-c_b}
 \geq ab=A_0^{r+s}>A_0,
\]
where the inequality is equivalent to \(c_a+c_b\geq1\).
Thus every multiplicative translate of \(A_0^\Z\) meets \(W\), and the
same corner argument applies.
\end{proof}

\begin{proof}[Failure of the all-observers conclusion in Theorem~\ref{thm:boundary}]
Assume \(\log a/\log b\notin\mathbb Q\) and
\(\gamma=1-c_a-c_b>0\).  Choose a sufficiently large positive integer
\(T\) so that
\[
 T\gamma>\min(c_a,c_b)^2.
\]
Corollary~\ref{cor:common-anchor} says that the quotient
\[
 \frac{T+K_{b,n}}{T+K_{a,m}}
\]
is compact and nowhere dense.  This is precisely the slope image of
\(K_{a,m}\times K_{b,n}\) from \((-T,-T)\).  Hence the
all-observers property fails.
The explicit unbounded open set asserted in part~(ii) is supplied by
Theorem~\ref{thm:open-arms} below.
\end{proof}

\begin{remark}\label{rem:egrs}
For these initial consecutive blocks, \(c_a+c_b=1\) is also the
numerical threshold in the two-base restricted-digit recurrence theorem
of Erd\H{o}s--Graham--Ruzsa--Straus \cite[Theorem~1]{EGRS}.  The
conclusions are different: the present radial-interior statement follows
from a local Astels argument, while the EGRS theorem is arithmetic.
\end{remark}

The boundary is not a Hausdorff-dimension boundary.  For example, put
\[
 a_r=6^r,\qquad m_r=\frac{2(6^r-1)}5,
\]
\[
 b_s=11^s,\qquad n_s=\frac{2(11^s-1)}5.
\]
Then \(c_{a_r}=c_{b_s}=2/5\), the bases are multiplicatively
independent, and
\[
 \Hd K_{a_r,m_r}\longrightarrow1,
 \qquad
 \Hd K_{b_s,n_s}\longrightarrow1.
\]
Nevertheless the all-observers property fails for every \(r,s\).
Thus no fixed dimension threshold below two replaces the thickness
boundary within this family.

\section{Open arms and dense rational bad observers}
\label{sec:open-arms}

The thin-side witness above is arithmetic.  We next show both that bad
observers occupy an open region and that arithmetic bad observers occur
arbitrarily close to the four corners of the product.  Assume throughout
this section that
\[
 \gamma=1-c_a-c_b>0,
 \qquad
 \frac{\log a}{\log b}\notin\mathbb Q.
\]
For \(A,B>0\), define the logarithmic slope image from \((-A,-B)\):
\[
 \Lambda_{A,B}
 =
 \left\{
 \log_b\frac{B+v}{A+u}:
 u\in K_{a,m},\ v\in K_{b,n}
 \right\}.
\]

\begin{theorem}[Open bad-observer arms]\label{thm:open-arms}
Set
\[
 A_*=\frac{c_a(c_a+c_b)}{2\gamma},
 \qquad
 B_*=\frac{c_b(c_a+c_b)}{2\gamma}.
\]
If \(A>A_*\) or \(B>B_*\), then
\(\Lambda_{A,B}\), and hence the radial image from \((-A,-B)\),
is compact and nowhere dense.  Consequently the bad-observer set
contains
\begin{equation}\label{eq:open-arms}
 \{(x,y):\operatorname{dist}(x,[0,c_a])>A_*,\ y\notin[0,c_b]\}
 \ \cup\
 \{(x,y):x\notin[0,c_a],\ \operatorname{dist}(y,[0,c_b])>B_*\}.
\end{equation}
\end{theorem}

\begin{proof}
We prove the assertion under \(A>A_*\); the other case follows by
interchanging the coordinates and taking reciprocal slopes.
Put
\[
 T=\frac{\gamma}{c_a}.
\]
For \(t\geq0\), the condition
\[
 s+\theta+tu\in(c_b,1)
 \qquad
 (\theta,u\in[0,c_a])
\]
is equivalent to
\[
 c_b<s<1-c_a-c_at.
\]
As \(0\leq t<T\), the possible residues \(s-At\) fill
\[
 W_+
 =
 \left(c_b-\frac{A\gamma}{c_a},1-c_a\right).
\]
Indeed, the feasible region in the \((t,s)\)-plane is convex, so its
projection under \((t,s)\mapsto s-At\) is an interval; the displayed
endpoints follow by letting \(t\downarrow0\) and \(t\uparrow T\).
For \(t=-z\leq0\), the same containment is equivalent to
\[
 c_b+c_az<s<1-c_a,
\]
and \(0\leq z<T\) gives
\[
 W_-
 =
 \left(c_b,1-c_a+\frac{A\gamma}{c_a}\right).
\]
The same convex-projection argument gives this whole interval.
Thus the available residue interval is
\begin{equation}\label{eq:residue-window}
 W
 =
 \left(
 c_b-\frac{A\gamma}{c_a},
 1-c_a+\frac{A\gamma}{c_a}
 \right),
\end{equation}
of length
\[
 |W|=\gamma\left(1+\frac{2A}{c_a}\right).
\]
The condition \(A>A_*\) is exactly \(|W|>1\).  Therefore for every
residue \(R\in\R/\Z\), there are \(s,t\), with \(|t|<T\), such that
\begin{equation}\label{eq:residue-fit}
 s-At\equiv R\pmod1,
 \qquad
 s+\theta+tu\in(c_b,1)
\end{equation}
for all \(\theta,u\in[0,c_a]\).

Write
\[
 \alpha=\log_ba,\qquad L=\log_b(B/A).
\]
For each fixed \(j\in\Z\) and all sufficiently large \(p\), put
\[
 q_{p,j}=\lfloor p\alpha-L\rfloor+j.
\]
Apply \eqref{eq:residue-fit} to
\[
 R_{p,j}=a^pA-b^{q_{p,j}}B\pmod1
\]
and choose corresponding \(s_{p,j},t_{p,j}\).  The choices satisfy
\[
 |t_{p,j}|<T
\]
uniformly in \(p\).  Define
\[
 x_{p,j}=\log_b(a^p+t_{p,j})-q_{p,j}.
\]
We claim that \(x_{p,j}\notin\Lambda_{A,B}\).  Otherwise, for some
\(u\in K_{a,m}\), \(v\in K_{b,n}\),
\[
 (a^p+t_{p,j})(A+u)=b^{q_{p,j}}(B+v).
\]
Let
\[
 \theta=\{a^pu\}\in K_{a,m},
 \qquad
 \zeta=\{b^{q_{p,j}}v\}\in K_{b,n}.
\]
Taking fractional parts yields
\[
 \zeta
 \equiv
 s_{p,j}+\theta+t_{p,j}u
 \pmod1.
\]
The right side has a representative in \((c_b,1)\) by
\eqref{eq:residue-fit}, while \(\zeta\in[0,c_b]\), a contradiction.

Finally, uniformly over the choices of \(t_{p,j}\),
\[
 x_{p,j}
 =
 L+\{p\alpha-L\}-j+o(1).
\]
For each fixed \(j\), every open subinterval of
\((L-j,L+1-j)\) contains arbitrarily late forbidden points.  If
\(\Lambda_{A,B}\) contained an open interval,
one could choose a smaller interval avoiding the boundaries
\(L+\Z\) and lying in one such unit strip; a forbidden point would
then enter it.  Thus \(\Lambda_{A,B}\) has empty interior.  It is
compact, hence nowhere dense.

It remains to justify the full observer region in
\eqref{eq:open-arms}.  Put \(P=K_{a,m}\times K_{b,n}\).  By
\eqref{eq:geometry} and its \(b\)-analogue, \(P\) is invariant under
the coordinate reflections
\[
 \sigma_1(u,v)=(c_a-u,v),
 \qquad
 \sigma_2(u,v)=(u,c_b-v).
\]
If \(Q_i\) is the linear part of \(\sigma_i\), then, with the usual
omission of an observer belonging to \(P\),
\[
 \Pi_{\sigma_i(O)}(\sigma_i(z))=Q_i\Pi_O(z)
 \qquad(z\in P\setminus\{O\}).
\]
Thus compactness and nowhere denseness of the radial image are invariant
under both reflections.  The southwest assertion follows by writing
\(O=(-A,-B)\), and its images under the group generated by
\(\sigma_1,\sigma_2\) are exactly \eqref{eq:open-arms}.
\end{proof}

\begin{proposition}[Dense rational bad observers]
\label{prop:dense-rational-observers}
Define
\[
 \mathcal A_a
 =
 \left\{
 \frac{N}{a^P}:
 N\in\Z_{\geq1},\ P\in\Z_{\geq0}
 \right\},
 \qquad
 \mathcal B_b
 =
 \left\{
 \frac{M+1-c_b}{b^Q}:
 M,Q\in\Z_{\geq0}
 \right\}.
\]
For every \(A\in\mathcal A_a\) and \(B\in\mathcal B_b\), the radial
image of \(K_{a,m}\times K_{b,n}\) from \((-A,-B)\) is compact and
nowhere dense.  Consequently, rational bad observers are dense in each
of the four components of
\[
 \bigl(\R\setminus[0,c_a]\bigr)
 \times
 \bigl(\R\setminus[0,c_b]\bigr).
\]
In particular, they occur arbitrarily close to each corner of
\([0,c_a]\times[0,c_b]\).
\end{proposition}

\begin{proof}
Fix \(A=N/a^P\in\mathcal A_a\) and
\(B=(M+1-c_b)/b^Q\in\mathcal B_b\).  For every \(p\geq P\),
\[
 \{a^pA\}=0.
\]
Moreover,
\[
 1-c_b=\frac{b-1-n}{b-1},
\]
so for every \(q\geq Q\),
\[
 b^qB
 =
 b^{q-Q}M+b^{q-Q}(1-c_b)
 \equiv1-c_b\pmod1,
\]
because \(b^{q-Q}-1\) is divisible by \(b-1\).  Since
\(0<1-c_b<1\), this gives
\[
 \{b^qB\}=1-c_b.
\]

Apply Theorem~\ref{thm:affine} to
\[
 E=A+K_{a,m},
 \qquad
 F=B+K_{b,n}.
\]
The stabilized anchor-residue difference has the lift
\[
 \delta=0-(1-c_b)=c_b-1.
\]
Choose
\[
 0<e<\frac{\gamma}{A+c_a}
\]
and take \(k=-1\).  Since \(e>0\), the endpoints in
\eqref{eq:affine-hull} are
\[
 L_e=c_b-1+eA
\]
and
\[
 U_e
 =c_b-1+c_a+e(A+c_a)
 =-\gamma+e(A+c_a).
\]
Therefore
\[
 -1+c_b<L_e\leq U_e<0.
\]
This is exactly the lifted principal-gap condition
\eqref{eq:affine-window}.  Theorem~\ref{thm:affine} shows that
\[
 \frac{B+K_{b,n}}{A+K_{a,m}}
\]
is compact and nowhere dense.  This quotient is the positive slope
image from \((-A,-B)\), and the slope coordinate is a homeomorphism
onto the corresponding arc of \(S^1\).  Hence the radial image has the
same properties.

Both \(\mathcal A_a\) and \(\mathcal B_b\) are dense in
\(\R_{>0}\): at levels \(P\) and \(Q\) they are translated grids with
meshes \(a^{-P}\) and \(b^{-Q}\), respectively.  Thus the displayed
family gives a countable dense set of rational bad observers in the
southwest quadrant.

Finally, the coordinate reflections \(\sigma_1,\sigma_2\) from the
proof of Theorem~\ref{thm:open-arms} preserve the product and the
topological type of each radial image.  They send \((-A,-B)\) to
\[
 (-A,-B),\quad
 (c_a+A,-B),\quad
 (-A,c_b+B),\quad
 (c_a+A,c_b+B).
\]
Since \(c_a,c_b\in\mathbb Q\), rationality is preserved.  The four
reflected families are dense in the four asserted components.
\end{proof}

\begin{remark}[Integer observers]\label{rem:integer-observers}
If \(A,B\in\Z_{>0}\), then
\(a^pA-b^qB\equiv0\pmod1\), so only the residue \(0\), rather than
the whole circle, must lie in \eqref{eq:residue-window}.  On the
\(A\)-side this happens if either
\[
 A>\frac{c_ac_b}{\gamma}
 \qquad\text{or}\qquad
 A>\frac{c_a^2}{\gamma}.
\]
Taking the cheaper alternative, and arguing symmetrically on the
\(B\)-side, gives the improved sufficient condition
\[
 A>\frac{c_a\min(c_a,c_b)}{\gamma}
 \quad\text{or}\quad
 B>\frac{c_b\min(c_a,c_b)}{\gamma}.
\]
At \(A=B=1\), this is
\[
 c_a+c_b+\min(c_a,c_b)^2<1.
\]
\end{remark}

For \(K_{3,1}\times K_{5,1}\), one has
\[
 A_*=\frac34,\qquad B_*=\frac38.
\]
Hence \eqref{eq:open-arms} gives a concrete open set of bad observers.
At \(A=B=1\), both threshold inequalities are strict, so
Theorem~\ref{thm:open-arms} also gives a second proof of
Corollary~\ref{cor:small-counterexample}, independent of the
common-anchor corollary.

\begin{corollary}[An explicit bad-observer ball]
\label{cor:k35-bad-ball}
Every observer in the open Euclidean ball
\[
 B_{\R^2}\left((-1,-1),\frac58\right)
\]
sees a compact nowhere-dense radial image of
\(K_{3,1}\times K_{5,1}\).
\end{corollary}

\begin{proof}
If \(O=(x,y)\) belongs to the displayed ball, then
\[
 x<-\frac38<0,\qquad y<-\frac38.
\]
Thus \(x\notin[0,1/2]\) and
\[
 \operatorname{dist}\left(y,\left[0,\frac14\right]\right)
 =-y>\frac38=B_*.
\]
The second arm in Theorem~\ref{thm:open-arms} applies.
\end{proof}

There is also an open family of positive-measure nowhere-dense images.
For \(K_{101,40}\times K_{103,40}\),
\[
 c_a=\frac25,\qquad c_b=\frac{20}{51},
 \qquad
 A_*=\frac{202}{265},\qquad
 B_*=\frac{2020}{2703}.
\]
Both Hausdorff dimensions exceed \(4/5\).  Banaji--Yu
\cite[Theorem~4.12]{BanajiYu} therefore gives positive angular measure
at each fixed observer, while Theorem~\ref{thm:open-arms} gives
nowhere denseness throughout \eqref{eq:open-arms}.  The positive-measure
conclusion is pointwise in the observer; no uniform lower bound is
asserted.

As \(c_a+c_b\uparrow1\), the pair of open-arm thresholds cannot remain
bounded, since
\[
 A_*+B_*=
 \frac{(c_a+c_b)^2}{2(1-c_a-c_b)}\longrightarrow\infty.
\]
If both \(c_a\) and \(c_b\) stay bounded away from zero, then each
threshold tends to infinity.  For every fixed pair below the boundary,
however,
Proposition~\ref{prop:dense-rational-observers} still gives arithmetic
bad observers arbitrarily close to the four corners.  At and above the
boundary, every observer sees interior.  Theorem~\ref{thm:open-arms}
and Proposition~\ref{prop:dense-rational-observers} provide explicit
subsets of the bad-observer set, not a characterization of that set.

\section{Fat division sets and product intervals}
\label{sec:contrast}

We now prove Theorem~\ref{thm:contrast-intro}.  The sets in
\eqref{eq:contrast-family-intro} can be written
\begin{equation}\label{eq:contrast-affine}
 E_r=\frac{13}{24}+K_{A_r,11q_r},
 \qquad
 F_s=\frac1{24}+K_{B_s,11t_s}.
\end{equation}
Both residual hulls have length \(11/24\), while the anchor residues
differ by \(1/2\).  In Theorem~\ref{thm:affine}, take \(e=0\).  The
source lies in
\[
 \frac12+K_{A_r,11q_r}
 \subset\left[\frac12,\frac{23}{24}\right],
\]
whereas the target lies in
\[
 K_{B_s,11t_s}\subset\left[0,\frac{11}{24}\right].
\]
The two intervals are disjoint.  Since \(25^r\) and \(49^s\) are
multiplicatively independent,
\begin{equation}\label{eq:division-nowhere}
 F_s/E_r\quad\text{is compact and nowhere dense}
\end{equation}
for every \(r,s\geq1\).

The number of digits in either block with base \(N\) is
\[
 L_N=\frac{11(N-1)}{24}+1=\frac{11N+13}{24}.
\]
The elementary inequality \(L_{xy}\geq L_xL_y\) follows from
\[
 L_{xy}-L_xL_y
 =
 \frac{143(x-1)(y-1)}{576}\geq0,
\]
and gives
\[
 \Hd E_r\geq\frac{\log12}{\log25}=0.771979\ldots,
\qquad
 \Hd F_s\geq\frac{\log23}{\log49}=0.805662\ldots.
\]
Both bounds exceed
\[
 \vartheta=\frac{\sqrt{65}-5}{4}=0.765564\ldots.
\]
Banaji--Yu \cite[Theorem~4.12]{BanajiYu}, applied at the origin, gives
positive angular measure for the radial image of \(E_r\times F_s\).
Since
\[
 E_r\subset[13/24,1],
 \qquad
 F_s\subset[1/24,1/2],
\]
angle and slope are bi-Lipschitz on the relevant compact arc.  Therefore
\begin{equation}\label{eq:division-positive}
 |F_s/E_r|>0.
\end{equation}
The quotient is perfect: with \(x\in E_r\) fixed, approximate any
\(y\in F_s\) by distinct \(y_k\in F_s\), so
\(y_k/x\to y/x\).  Equations
\eqref{eq:division-nowhere}--\eqref{eq:division-positive} make
\(F_s/E_r\) a fat Cantor set.

The exact dimensions are
\[
 \Hd E_r=\frac{\log L_{A_r}}{\log A_r},
 \qquad
 \Hd F_s=\frac{\log L_{B_s}}{\log B_s},
\]
and both tend to one.

It remains to justify the self-product statement.  We isolate the
external input from Yu's proof.

\begin{proposition}[Product-interior consequence of Yu's argument]
\label{prop:product-interior}
Let \(K=K_{N,D}\Subset(0,\infty)\) be a unit-interval missing-digit
set and let \(\mu\) be its natural uniform Bernoulli measure.  If
\[
 \Ldim\mu>\frac34,
\]
then \(K\cdot K\) contains a nondegenerate interval.
\end{proposition}

\begin{proof}
For \(\lambda=\mu\otimes\mu\),
\[
 \widehat\lambda(\xi,\eta)
 =
 \widehat\mu(\xi)\widehat\mu(\eta),
\]
and \(\lambda\) is the planar base-\(N\) missing-digit measure with
digit set \(D\times D\).  Thus the definition of Fourier
\(l^1\)-dimension gives the product
inequality
\[
 \Ldim\lambda\geq2\Ldim\mu>\frac32.
\]
The hyperbolas
\[
 H_z=\{(x,y)\in(0,\infty)^2:xy=z\}
\]
have curvature exponent \(1/2\), so the threshold in
\cite[Theorem~2.5]{YuProduct} is \(2-1/2=3/2\).
Because \(K\Subset(0,\infty)\), the localization \(Y\) used in the proof
of Theorem C in \cite{YuProduct} may be taken to be \(K\times K\)
itself; no translation or rotation is required.  Put
\[
 \Sigma=
 \left\{
 \frac{x}{\sqrt{x_1x_2}}:x=(x_1,x_2)\in K\times K
 \right\}\subset H_1.
\]
Choose a nonnegative \(\rho\in C_c^\infty(H_1)\) that is strictly
positive on a neighbourhood of \(\Sigma\), and choose
\(U\Subset(0,\infty)\) whose interior contains
\[
 \left\{\sqrt{x_1x_2}:x=(x_1,x_2)\in K\times K\right\}.
\]
These compactness choices make the curvature and transversality constants
uniform.  They also verify the condition called \emph{(Positive)} in
Section 5.7, Case 3, of \cite{YuProduct}: uniformly for
\(x\in K\times K\) and all sufficiently small \(\varepsilon>0\),
\[
 \left|
 \left\{t\in U:
 \operatorname{dist}\bigl(x,t\operatorname{supp}\rho\bigr)
 <\varepsilon\right\}
 \right|
 \gtrsim\varepsilon.
\]
Indeed, \(t=\sqrt{x_1x_2}\) places \(x\) on \(t\Sigma\), and a uniform
interval of \(O(\varepsilon)\)-perturbations of \(t\) remains within
\(O(\varepsilon)\) of \(x\).  Theorems 2.5 and 2.10 and the cited Case 3
therefore give a continuous incidence function \(I(t)\) on \(U\), as in
the proof of Theorem C, with
\[
 \int_U I(t)\,dt>0.
\]
Hence \(I\) is positive on a nonempty open interval.  Positivity at
\(t\) implies \(t\operatorname{supp}\rho\cap(K\times K)\neq\varnothing\).
Since \(tH_1=H_{t^2}\), the corresponding interval of \(t^2\)-values
lies in \(K\cdot K\).
\end{proof}

This proposition is extracted from the argument proving Yu's Theorem C;
it is not the literal statement of that theorem.  For a one-dimensional
consecutive digit block \(D\) in base \(N\),
\cite[Theorem~2.15(2)]{YuProduct} gives
\begin{equation}\label{eq:l1-bound}
 \Ldim\mu_{N,D}
 \geq
 \frac{\log|D|-\log(2\log N)}{\log N}.
\end{equation}
The block position does not enter the estimate.  In the present family
the right side tends to one, so
Proposition~\ref{prop:product-interior} applies for all sufficiently
large \(r\) and \(s\).

The explicit estimate in \cite{YuProduct} is already sufficient at
\[
 r=6,\qquad s=5.
\]
For \(N=25^6\), the lower bound in \eqref{eq:l1-bound} is
\(0.770411\ldots\); for \(N=49^5\), it is
\(0.771745\ldots\).  Thus
\[
 E_6E_6\quad\text{and}\quad F_5F_5
\]
contain nondegenerate intervals.  This proves
Theorem~\ref{thm:contrast-intro}.

\section{Discussion}

The results above are deliberately limited to conclusions supported by
the two mechanisms used in the proofs.

\begin{itemize}
\item The condition \(c_a+c_b=1\) is an exact all-observers boundary
only within the multiplicatively independent initial-block family.
The thin side for multiplicatively dependent bases is not classified.
\item The affine-anchor theorem and the open-arm theorem are sufficient
modular obstructions.  The former produces the dense arithmetic family
in Proposition~\ref{prop:dense-rational-observers}, whereas the latter
produces robust open regions.  Internal Cantor gaps may produce
additional bad parameters beyond the hull windows used here.
\item The two obstructions are genuinely distinct.  For the family in
Section~\ref{sec:contrast}, \(c_a=c_b=11/24\), so
\(A_*=B_*=121/48\), whereas the observer in residual coordinates is
\((-13/24,-1/24)\); hence Theorem~\ref{thm:open-arms} does not apply.
Here nowhere denseness instead comes from the residue offset
\(\delta=1/2\), which for \(e=0\) places \([1/2,23/24]\) strictly
inside the principal gap \((11/24,1)\).
\item The exact bad-observer set below the thickness boundary is not
determined.  Theorem~\ref{thm:open-arms} supplies an explicit open
subset, and Proposition~\ref{prop:dense-rational-observers} supplies a
dense arithmetic subset of the four exterior corner regions.  The side
strips and the behaviour near non-corner points of the product remain
unclassified.
\item Proposition~\ref{prop:product-interior} is used only for the two
self-products.  No cross-product interval statement for \(E_rF_s\) is
made.
\end{itemize}

The distinction between product and division is structural.  A division
fibre is the line \(y=zx\); after cross multiplication, digit shifting
compares an affine tail with a fixed forbidden modular window.  A product
fibre \(xy=z\) is curved and belongs to a scaled family of hyperbolas,
to which Yu's group-action/Fourier argument applies.  Passing to
logarithmic coordinates linearizes the product equation but destroys the
original affine missing-digit self-similarity.

\begingroup
\footnotesize
\paragraph{Acknowledgements.}
The author thanks Han Yu for helpful correspondence, for raising the
two-set and observer-set questions, and for pointing out the relevance
of the Erd\H{o}s--Graham--Ruzsa--Straus method.

\paragraph{Disclosure of AI assistance.}
The author used OpenAI's ChatGPT and Codex extensively in the development
and preparation of this manuscript, including for mathematical
exploration, proof development and checking, computational and citation
checks, editorial revision, and \TeX/PDF preparation.  Anthropic's Claude
was also used, through materials and queries supplied by the author, to
provide additional detailed checks of proofs, constants, citations, and exposition
across multiple drafts.  The author determined which suggestions to adopt
and assumes responsibility for the final manuscript and any remaining
errors.
\endgroup


\begingroup
\footnotesize
\begin{thebibliography}{9}
\setlength{\itemsep}{0pt}

\bibitem{ART}
J.~S. Athreya, B.~Reznick and J.~T. Tyson,
\emph{Cantor set arithmetic},
Amer. Math. Monthly \textbf{126} (2019), no.~1, 4--17.
\href{https://doi.org/10.1080/00029890.2019.1528121}
{\texttt{doi:10.1080/00029890.2019.1528121}}.

\bibitem{Astels}
S.~Astels,
\emph{Cantor sets and numbers with restricted partial quotients},
Trans. Amer. Math. Soc. \textbf{352} (2000), no.~1, 133--170.

\bibitem{BanajiYu}
A.~Banaji and H.~Yu,
\emph{Fourier transform of nonlinear images of self-similar measures:
quantitative aspects},
arXiv:2503.07508v2 (2026), to appear in Peking Math. J.
\href{https://arxiv.org/abs/2503.07508}
{\texttt{arXiv:2503.07508}}.

\bibitem{EGRS}
P.~Erd\H{o}s, R.~L. Graham, I.~Z. Ruzsa and E.~G. Straus,
\emph{On the prime factors of \(\binom{2n}{n}\)},
Math. Comp. \textbf{29} (1975), no.~129, 83--92.
\href{https://doi.org/10.1090/S0025-5718-1975-0369288-3}
{\texttt{doi:10.1090/S0025-5718-1975-0369288-3}}.

\bibitem{YuRadial}
H.~Yu,
\emph{Fractal projections with an application in number theory},
Ergodic Theory Dynam. Systems \textbf{43} (2023), no.~5, 1760--1784.
\href{https://doi.org/10.1017/etds.2022.2}
{\texttt{doi:10.1017/etds.2022.2}}.

\bibitem{YuProduct}
H.~Yu,
\emph{Missing digits points near manifolds},
arXiv:2309.00130v1 (2023).
\href{https://arxiv.org/abs/2309.00130}
{\texttt{arXiv:2309.00130}}.

\end{thebibliography}
\endgroup
\end{document}